\documentclass[a4paper,11pt]{amsart}
\usepackage{amsmath,amssymb,amsfonts,latexsym}

\newtheorem{theorem}{Theorem}[section]
\newtheorem{lemma}[theorem]{Lemma}

\input{tcilatex}

\begin{document}
\title[\textbf{Spectral properties of a Sturm-Liouville problem}]{\textbf{On
spectral properties of a Sturm-Liouville problem with transmission
conditions and eigenparameter dependent boundary conditions}}
\author{\textbf{Serkan Araci}}
\address{\textbf{Gaziantep University, Faculty of Science and Arts,
Department of Mathematics, 27310 Gaziantep Turkey}}
\email{\textbf{mtsrkn@hotmail.com}}
\author{\textbf{Mehmet Acikgoz}}
\address{\textbf{Gaziantep University, Faculty of Science and Arts,
Department of Mathematics, 27310 Gaziantep Turkey}}
\email{\textbf{acikgoz@gantep.edu.tr}}
\author{\textbf{Azad Bayramov}}
\address{\textbf{Department of Mathematics Education, Faculty of Education,
Recep Tayyip Erdogan University, Rize, Turkey.}}
\author{\textbf{Erdo\u{g}an \c{S}en}}
\address{\textbf{Namik Kemal University, Faculty of Arts and Science,
Department of Mathematics 59030 Tekirda\u{g}, Turkey}}
\email{\textbf{erdogan.math@gmail.com}}

\begin{abstract}
In this work, we consider not only a discontinuous boundary-value problem
with retarded argument and four supplementary transmission conditions at the
two points of discontinuities but also, eigenparameter-dependent boundary
conditions and obtain asymptotic formulas for the eigenvalues and
corresponding eigenfunctions.

\vspace{2mm}\noindent \textsc{2010 Mathematics Subject Classification.}
34L20, 35R10.

\vspace{2mm}

\noindent \textsc{Keywords and phrases.} Differential equation with retarded
argument; eigenparameter; transmission conditions; asymptotics of
eigenvalues and eigenfunctions.
\end{abstract}

\thanks{}
\maketitle




\section{\textbf{Introduction}}


In recent years the Sturm-Liouville problems with transmission conditions
and/or with eigenparameter-dependent boundary conditions have been an
important research topic in mathematical physics [1--12]. Boundary value
problems with discontinuity conditions inside the interval often appear in
applications. Such problems are connected with discontinuous material
properties, such as heat and mass transfer, varied assortment of physical
transfer problems, vibrating string problems when the string is loaded
additionally with point masses and diffraction problems. Sturm-Liouville
problems with eigenparameter-dependent boundary conditions are connected
with heat conduction problems and vibrating string problems and so on. It
must be noted that some problems with transmission conditions which arise in
mechanics (thermal condition problem for a thin laminated plate) were
studied in \cite{Ti}.

In this paper we investigated the eigenvalues and eigenfunctions of a
discontinuous boundary value problem with retarded argument. Namely, we
consider the boundary value problem for the differential equation%
\begin{equation}
y^{\prime \prime }(x)+q(x)y(x-\Delta (x))+\mu ^{2}y(x)=0  \label{equation 1}
\end{equation}%
on $\left[ 0,h_{1}\right) \cup \left( h_{1},h_{2}\right) \cup \left(
h_{2},\pi \right] $, with boundary conditions

\begin{equation}
\mu y(0)+y^{\prime }(0)=0\text{,}  \label{equation 2}
\end{equation}%
\begin{equation}
\mu ^{2}y(\pi )+y^{\prime }(\pi )=0\text{,}  \label{equation 3}
\end{equation}%
and transmission conditions%
\begin{equation}
y(h_{1}-0)-\delta y(h_{1}+0)=0\text{,}  \label{equation 4}
\end{equation}%
\begin{equation}
y^{\prime }(h_{1}-0)-\delta y^{\prime }(h_{1}+0)=0\text{,}
\label{equation 5}
\end{equation}%
\begin{equation}
y(h_{2}-0)-\theta y(h_{2}+0)=0\text{,}  \label{equation 6}
\end{equation}%
\begin{equation}
y^{\prime }(h_{2}-0)-\theta y^{\prime }(h_{2}+0)=0\text{,}
\label{equation 7}
\end{equation}%
where the real-valued function $q(x)$ is continuous in $\left[
0,h_{1}\right) \cup \left( h_{1},h_{2}\right) \cup \left( h_{2},\pi \right]
~ $and has finite limits 
\begin{equation*}
q(h_{1}\pm 0)=\lim_{x\rightarrow h_{1}\pm 0}q(x),\text{ }q(h_{2}\pm
0)=\lim_{x\rightarrow h_{2}\pm 0}q(x)\text{,}
\end{equation*}%
the real valued function $\Delta (x)\geq 0$ continuous in $\left[
0,h_{1}\right) \cup \left( h_{1},h_{2}\right) \cup \left( h_{2},\pi \right] $
and has finite limits%
\begin{equation*}
\Delta (h_{1}\pm 0)=\lim_{x\rightarrow h_{1}\pm 0}\Delta (x),\text{ }\Delta
(h_{2}\pm 0)=\lim_{x\rightarrow h_{2}\pm 0}\Delta (x);
\end{equation*}%
$x-\Delta (x)\geq 0$ \textit{if}$\ \ x\in \left[ 0,h_{1}\right) $; $x-\Delta
(x)\geq h_{1}$, \textit{if} $x\in \left( h_{1},h_{2}\right) $; $x-\Delta
(x)\geq h_{2}$, \textit{if} $x\in \left( h_{2},\pi \right) $; $\mu $ is a
real positive eigenparameter; $h_{1},h_{2},\delta ,\theta \neq 0$ are
arbitrary real numbers such that $0<h_{1}<h_{2}<\pi $ and $a_{2}\neq 0$.

Let $w_{1}(x,\lambda )$ be a solution of Eq. (\ref{equation 1}) on $\left[
0,h_{1}\right] ,$ satisfying the initial conditions%
\begin{equation}
w_{1}\left( 0,\mu \right) =1\text{ and }w_{1}^{\prime }\left( 0,\mu \right)
=-\mu \text{.}  \label{equation 8}
\end{equation}%
The conditions (\ref{equation 8}) define a unique solution of Eq. (\ref%
{equation 1}) on $\left[ 0,h_{1}\right] $ (\cite{Norkin 2}, p. 12).

After defining the above solution, then we shall define the solution $%
w_{2}\left( x,\mu \right) $ of Eq. (\ref{equation 1}) on $\left[ h_{1},h_{2}%
\right] $ by means of the solution $w_{1}\left( x,\mu \right) $ using the
initial conditions%
\begin{equation}
w_{2}\left( h_{1},\mu \right) =\delta ^{-1}w_{1}\left( h_{1},\mu \right) 
\text{ and}\quad w_{2}^{\prime }(h_{1},\>\mu )=\delta ^{-1}w_{1}^{\prime
}(h_{1},\>\mu )\text{.}  \label{equation 9}
\end{equation}%
The conditions (\ref{equation 9}) define a unique solution of Eq. (\ref%
{equation 1}) on $\left[ h_{1},h_{2}\right] .$

After describing the above solution, then we shall give the solution $%
w_{3}\left( x,\mu \right) $ of Eq. (\ref{equation 1}) on $\left[ h_{2},\pi %
\right] $ by means of the solution $w_{2}\left( x,\mu \right) $ using the
initial conditions%
\begin{equation}
w_{3}\left( h_{2},\mu \right) =\theta ^{-1}w_{2}\left( h_{2},\mu \right) 
\text{ \textit{and}}\quad w_{3}^{\prime }(h_{2},\mu )=\theta
^{-1}w_{2}^{\prime }(h_{2},\mu )\text{.}  \label{equation 10}
\end{equation}%
The conditions (\ref{equation 10}) define a unique solution of Eq. (\ref%
{equation 1}) on $\left[ h_{2},\pi \right] .$

Consequently, the function $w\left( x,\mu \right) $ is defined on $\left[
0,h_{1}\right) \cup \left( h_{1},h_{2}\right) \cup \left( h_{2},\pi \right] $
by the equality%
\begin{equation*}
w(x,\lambda )=\left\{ 
\begin{array}{cc}
w_{1}(x,\mu ), & x\in \lbrack 0,h_{1}), \\ 
w_{2}(x,\mu ), & x\in \left( h_{1},h_{2}\right) , \\ 
w_{3}(x,\mu ), & x\in \left( h_{2},\pi \right]%
\end{array}%
\right.
\end{equation*}%
is a solution of the Eq. (\ref{equation 1}) on $\left[ 0,h_{1}\right) \cup
\left( h_{1},h_{2}\right) \cup \left( h_{2},\pi \right] $; which satisfies
one of the boundary conditions and four transmission conditions.

\begin{lemma}
Let $w\left( x,\mu \right) $ be a solution of Eq. (\ref{equation 1})$.$ Then
the following integral equations hold:\ 
\begin{equation}
w_{1}(x,\mu )=-\sin \mu x+\cos \mu x-\frac{1}{\mu }\int\limits_{0}^{{x}%
}q\left( \tau \right) \sin \left( \mu \left( x-\tau \right) \right)
w_{1}\left( \tau -\Delta \left( \tau \right) ,\mu \right) d\tau \text{,}
\label{equation 11}
\end{equation}%
\begin{align}
w_{2}(x,\mu )& =\frac{1}{\delta }w_{1}\left( h_{1},\mu \right) \cos \mu
\left( x-h_{1}\right) +\frac{w_{1}^{\prime }\left( h_{1},\lambda \right) }{%
\mu \delta }\sin \mu \left( x-h_{1}\right)  \notag \\
& -\frac{1}{\mu }\int\limits_{h_{1}}^{{x}}q\left( \tau \right) \sin s\left(
x-\tau \right) w_{2}\left( \tau -\Delta \left( \tau \right) ,\mu \right)
d\tau \text{,}  \label{equation 12}
\end{align}%
\begin{align}
w_{3}(x,\mu )& =\frac{1}{\theta }w_{2}\left( h_{2},\mu \right) \cos \mu
\left( x-h_{2}\right) +\frac{w_{2}^{\prime }\left( h_{2},\mu \right) }{\mu
\theta }\sin \mu \left( x-h_{2}\right)  \notag \\
& -\frac{1}{\mu }\int\limits_{h_{2}}^{{x}}q\left( \tau \right) \sin \mu
\left( x-\tau \right) w_{3}\left( \tau -\Delta \left( \tau \right) ,\mu
\right) d\tau \text{.}  \label{equation 13}
\end{align}
\end{lemma}

\begin{proof}
To prove this lemma, it is enough to substitute $\>-s^{2}w_{1}(\tau ,\mu
)-w_{1}^{\prime \prime }(\tau ,\mu ),\>-s^{2}w_{2}(\tau ,\mu )-w_{2}^{\prime
\prime }(\tau ,\mu )$ and $\>-s^{2}w_{3}(\tau ,\mu )-w_{3}^{\prime \prime
}(\tau ,\mu )\>$ instead of $\>-q(\tau )w_{1}(\tau -\Delta (\tau ),\mu
),\>-q(\tau )w_{2}(\tau -\Delta (\tau ),\mu )$ and $\>-q(\tau )w_{3}(\tau
-\Delta (\tau ),\mu )\>$ in the integrals in (\ref{equation 11}), (\ref%
{equation 12}) and (\ref{equation 13}) respectively and integrate by parts
twice.
\end{proof}

\begin{theorem}
The problem (\ref{equation 1})-(\ref{equation 7}) can have only simple
eigenvalues.
\end{theorem}

\begin{proof}
Let $\widetilde{\mu }$ be an eigenvalue of the problem (\ref{equation 1})-(%
\ref{equation 7}) and%
\begin{equation*}
\widetilde{y}(x,\widetilde{\mu })=\left\{ 
\begin{array}{cc}
\widetilde{y}_{1}(x,\widetilde{\mu }), & x\in \lbrack 0,h_{1}), \\ 
\widetilde{y}_{2}(x,\widetilde{\mu }), & x\in \left( h_{1},h_{2}\right) , \\ 
\widetilde{y}_{3}(x,\widetilde{\mu }), & x\in \left( h_{2},\pi \right]%
\end{array}%
\right.
\end{equation*}%
be a corresponding eigenfunction. Then, from (\ref{equation 2}) and (\ref%
{equation 8}), it follows that the determinant%
\begin{equation*}
W\left[ \widetilde{y}_{1}(0,\widetilde{\mu }),w_{1}(0,\widetilde{\mu })%
\right] =\left\vert 
\begin{array}{c}
\widetilde{y}_{1}(0,\widetilde{\mu })\text{ \ \ \ \ \ \ \ \ \ }1 \\ 
\widetilde{y}_{1}^{\prime }(0,\widetilde{\mu })\text{ \ \ \ \ \ \ }-%
\widetilde{\mu }%
\end{array}%
\right\vert =0\text{,}
\end{equation*}%
and, by Theorem 2.2.2, in \cite{Norkin 2} the functions $\widetilde{y}_{1}(x,%
\widetilde{\mu })$ and $w_{1}(x,\widetilde{\mu })$ are linearly dependent on 
$\left[ 0,h_{1}\right] $. We can also prove that the functions $\widetilde{y}%
_{2}(x,\widetilde{\mu })$ and $w_{2}(x,\widetilde{\mu })$ are linearly
dependent on $\left[ h_{1},h_{2}\right] $ and $\widetilde{y}_{3}(x,%
\widetilde{\mu })$ and $w_{3}(x,\widetilde{\mu })$ are linearly dependent on 
$\left[ h_{2},\pi \right] $. Hence%
\begin{equation}
\widetilde{y}_{i}(x,\widetilde{\mu })=K_{i}w_{i}(x,\widetilde{\mu })\text{ \
\ \ }\left( i=1,2,3\right)  \label{equation 14}
\end{equation}%
for some $K_{1}\neq 0,$ $K_{2}\neq 0$ and $K_{3}\neq 0$. We must show that $%
K_{1}=K_{2}$ and $K_{2}=K_{3}$. Suppose that $K_{2}\neq K_{3}$. From the
equalities (\ref{equation 6}) and (\ref{equation 14}), we have%
\begin{align*}
\widetilde{y}(h_{2}-0,\widetilde{\mu })-\theta \widetilde{y}(h_{2}+0,%
\widetilde{\mu })& =\widetilde{y_{2}}(h_{2},\widetilde{\mu })-\theta 
\widetilde{y_{3}}(h_{2},\widetilde{\mu }) \\
& =K_{2}w_{2}(h_{2},\widetilde{\mu })-\theta K_{3}w_{3}(h_{2},\widetilde{\mu 
}) \\
& =K_{2}\theta w_{3}(h_{2},\widetilde{\mu })-K_{3}\theta w_{3}(h_{2},%
\widetilde{\mu }) \\
& =\theta \left( K_{2}-K_{3}\right) w_{3}(h_{2},\widetilde{\mu })=0\text{.}
\end{align*}%
Since $\theta \left( K_{2}-K_{3}\right) \neq 0$ it follows that 
\begin{equation}
w_{3}\left( h_{2},\widetilde{\mu }\right) =0\text{.}  \label{equation 15}
\end{equation}%
By the same procedure from equality (\ref{equation 7}) we can derive that%
\begin{equation}
w_{3}^{\prime }\left( h_{2},\widetilde{\mu }\right) =0\text{.}
\label{equation 16}
\end{equation}%
From the fact that $w_{2}(x,\widetilde{\mu })$ is a solution of the
differential Eq. (\ref{equation 1}) on $\left[ h_{2},\pi \right] $ and
satisfies the initial conditions (\ref{equation 15}) and (\ref{equation 16})
it follows that $w_{3}(x,\widetilde{\mu })=0$ identically on $\left[
h_{2},\pi \right] $.

By using this method, we may also find%
\begin{align*}
w_{2}\left( h_{2},\widetilde{\mu }\right) & =w_{2}^{\prime }\left( h_{2},%
\widetilde{\mu }\right) =0\text{,} \\
w_{1}\left( h_{1},\widetilde{\mu }\right) & =w_{1}^{\prime }\left( h_{1},%
\widetilde{\mu }\right) =0\text{.}
\end{align*}%
From the latter discussions of $w_{3}(x,\widetilde{\mu })$ it follows that $%
w_{2}(x,\widetilde{\mu })=0$ and $w_{1}(x,\widetilde{\mu })=0$ identically
on $\left( h_{1},h_{2}\right) $ and $\left[ 0,h_{1}\right) .$ But this
contradicts (\ref{equation 8}), thus completing the proof.
\end{proof}

\section{\textbf{Existence theorem}}

The function $w(x,\>\mu )\>$ is defined in section $1$ is a nontrivial
solution of Eq. (\ref{equation 1}) satisfying conditions (\ref{equation 2})
and (\ref{equation 4})-(\ref{equation 7}). Putting$\>$ $w(x,\>\mu )\>$ into (%
\ref{equation 3}), we get the characteristic equation 
\begin{equation}
H(\lambda )\equiv w^{\prime }(\pi ,\>\mu )+\mu ^{2}w(\pi ,\>\mu )=0\text{.}
\label{equation 17}
\end{equation}

By Theorem 1 the set of eigenvalues of boundary-value problem (\ref{equation
1})-(\ref{equation 7}) coincides with the set of real roots of Eq. (\ref%
{equation 17}). Let 
\begin{equation*}
\>q_{1}=\int\limits_{0}^{h_{1}}|q(\tau )|d\tau
,q_{2}=\int\limits_{h_{1}}^{h_{2}}\left\vert q(\tau )\right\vert d\tau \text{
and }q_{3}=\int\limits_{h_{2}}^{{\pi }}\left\vert q(\tau )\right\vert d\tau 
\text{.}
\end{equation*}

\begin{lemma}
(1) Let $\mu \geq 2q_{1}$. Then for the solution $w_{1}\left( x,\mu \right) $
of Eq. (\ref{equation 11}), the following inequality holds$:$%
\begin{equation}
\left\vert w_{1}\left( x,\mu \right) \right\vert \leq 2\sqrt{2}\text{, \ \ }%
x\in \left[ 0,h_{1}\right] \text{.}  \label{equation 18}
\end{equation}%
(2) Let $\mu \geq \max \left\{ 2q_{1},2q_{2}\right\} $. Then for the
solution $w_{2}\left( x,\mu \right) $ of Eq. (\ref{equation 12}), the
following inequality holds$:$%
\begin{equation}
\left\vert w_{2}\left( x,\mu \right) \right\vert \leq \frac{8\sqrt{2}}{%
\left\vert \delta \right\vert }\text{,\ \ }x\in \left[ h_{1},h_{2}\right] 
\text{.}  \label{equation 19}
\end{equation}%
(3) Let $\mu \geq \max \left\{ 2q_{1},2q_{2},2q_{3}\right\} $. Then for the
solution $w_{3}\left( x,\mu \right) $ of Eq. (\ref{equation 13}), the
following inequality holds$:$%
\begin{equation}
\left\vert w_{3}\left( x,\mu \right) \right\vert \leq \frac{32\sqrt{2}}{%
\left\vert \delta \theta \right\vert }\text{,\ \ }x\in \left[ h_{2},\pi %
\right] \text{.}  \label{equation 20}
\end{equation}
\end{lemma}

\begin{proof}
Let $B_{1\mu }=\max_{\left[ 0,h_{1}\right] }\left\vert w_{1}\left( x,\mu
\right) \right\vert $. Then from (\ref{equation 11}), it follows that, for
every $\mu $, the following inequality holds:%
\begin{equation*}
B_{1\mu }\leq \sqrt{2}+\frac{1}{\mu }B_{1\mu }q_{1}\text{.}
\end{equation*}%
If $\mu \geq 2q_{1}$ we get (\ref{equation 18}). Differentiating (\ref%
{equation 11}) with respect to $x$, we have%
\begin{equation}
w_{1}^{\prime }(x,\mu )=-\mu \sin \mu x-\mu \cos \mu
x-\int\limits_{0}^{x}q(\tau )\cos \mu \left( x-\tau \right) w_{1}(\tau
-\Delta \left( \tau \right) ,\mu )d\tau \text{.}  \label{equation 21}
\end{equation}%
From expressions of (\ref{equation 21}) and (\ref{equation 18}), it follows
that, for $\mu \geq 2q_{1}$, the following inequality holds:%
\begin{equation}
\frac{\left\vert w_{1}^{\prime }(x,\mu )\right\vert }{\left\vert \mu
\right\vert }\leq 2\sqrt{2}\text{.}  \label{equation 22}
\end{equation}%
Let $B_{2\mu }=\max_{\left[ h_{1},h_{2}\right] }\left\vert w_{2}\left( x,\mu
\right) \right\vert $. Then from (\ref{equation 9}), (\ref{equation 18}) and
(\ref{equation 22}) it follows that, for $\mu \geq 2q_{1}$ and $\mu \geq
2q_{2}$, the following inequality holds:%
\begin{equation*}
B_{2\mu }\leq \frac{4\sqrt{2}}{\left\vert \delta \right\vert }+\frac{1}{\mu }%
B_{2\mu }q_{2}\text{.}
\end{equation*}%
Hence, if $\mu \geq \max \left\{ 2q_{1},2q_{2}\right\} $, it reduces to (\ref%
{equation 19}). Differentiating (\ref{equation 12}) with respect to $\>x$,
we get%
\begin{gather}
w_{2}^{\prime }(x,\mu )=-\frac{\mu }{\delta }w_{1}\left( h_{1},\mu \right)
\sin \mu \left( x-h_{1}\right) +\frac{w_{1}^{\prime }\left( h_{1},\mu
\right) }{\delta }\cos \mu \left( x-h_{1}\right)  \notag \\
-\int\limits_{h_{1}}^{{x}}q(\tau )\cos \mu \left( x-\tau \right) w_{2}(\tau
-\Delta \left( \tau \right) ,\mu )d\tau \text{.}  \label{equation 23}
\end{gather}%
From (\ref{equation 18}) and (\ref{equation 23}), it follows that, for $\mu
\geq 2q_{1}$ and $\mu \geq 2q_{2}$, the following inequality holds:%
\begin{equation}
\frac{\left\vert w_{2}^{\prime }(x,\mu )\right\vert }{\mu }\leq \frac{8\sqrt{%
2}}{\left\vert \delta \right\vert }\text{.}  \label{equation 24}
\end{equation}%
Let $B_{3\mu }=\max_{\left[ h_{2},\pi \right] }\left\vert w_{3}\left( x,\mu
\right) \right\vert $. Then from (\ref{equation 13}), (\ref{equation 19})
and (\ref{equation 24}) it follows that, for $\mu \geq 2q_{1},$ $\mu \geq
2q_{2}$ and $\mu \geq 2q_{3}$, the following inequality holds:%
\begin{equation*}
B_{3\mu }\leq \frac{16\sqrt{2}}{\left\vert \delta \theta \right\vert }+\frac{%
1}{\mu }B_{3\mu }q_{3}\text{.}
\end{equation*}%
Hence, if $\mu \geq \max \left\{ 2q_{1},2q_{2},2q_{3}\right\} $ we obtain (%
\ref{equation 20}).
\end{proof}

\begin{theorem}
The problem (\ref{equation 1})-(\ref{equation 7}) has infinitely number of
positive eigenvalues.
\end{theorem}

\begin{proof}
Differentiating (\ref{equation 13}) with respect to $\>x$, we readily see
that%
\begin{gather}
w_{3}^{\prime }(x,\mu )=-\frac{\mu }{\theta }w_{2}\left( h_{2},\mu \right)
\sin \mu \left( x-h_{2}\right) +\frac{w_{2}^{\prime }\left( h_{2},\mu
\right) }{\theta }\cos \mu \left( x-h_{2}\right)  \notag \\
-\int\limits_{h_{2}}^{{x}}q(\tau )\cos \mu \left( x-\tau \right) w_{3}(\tau
-\Delta \left( \tau \right) ,\mu )d\tau \text{.}  \label{equation 25}
\end{gather}%
With the helps of (\ref{equation 11}), (\ref{equation 12}), (\ref{equation
13}), (\ref{equation 17}), (\ref{equation 21}) and (\ref{equation 25}), we
have the following:%
\begin{equation*}
-\frac{\mu ^{2}}{\theta \delta }\left( \sin \mu \pi -\cos \mu \pi \right) -%
\frac{\mu }{\theta \delta }\int\limits_{0}^{h_{1}}q(\tau )\sin \mu (\pi
-\tau )w_{1}(\tau -\Delta (\tau ),\mu )d\tau
\end{equation*}%
\begin{equation*}
-\frac{\mu }{\theta }\int\limits_{h_{1}}^{h_{2}}q(\tau )\sin \mu (\pi -\tau
)w_{2}(\tau -\Delta (\tau ),\mu )d\tau
\end{equation*}%
\begin{equation*}
-\mu \int\limits_{h_{2}}^{\pi }q(\tau )\sin \mu (\pi -\tau )w_{3}(\tau
-\Delta (\tau ),\mu )d\tau
\end{equation*}%
\begin{equation*}
-\frac{\mu }{\theta \delta }\left( \sin \mu \pi +\cos \mu \pi \right) -\frac{%
1}{\theta \delta }\int\limits_{0}^{h_{1}}q(\tau )\cos \mu (\pi -\tau
)w_{1}(\tau -\Delta (\tau ),\mu )d\tau
\end{equation*}%
\begin{equation*}
-\frac{1}{\theta }\int\limits_{h_{1}}^{h_{2}}q(\tau )\cos \mu (\pi -\tau
)w_{2}(\tau -\Delta (\tau ),\mu )d\tau
\end{equation*}%
\begin{equation}
-\int\limits_{h_{2}}^{\pi }q(\tau )\cos \mu (\pi -\tau )w_{3}(\tau -\Delta
(\tau ),\mu )d\tau =0\text{.}  \label{equation 26}
\end{equation}%
Let $\>\mu \>$ be sufficiently big. Then, by (\ref{equation 18}), (\ref%
{equation 19}) and (\ref{equation 20}), Eq. (\ref{equation 26}) may be
rewritten in the form%
\begin{equation}
\mu \sin \left( \mu \pi -\frac{\pi }{4}\right) +O(1)=0\text{.}
\label{equation 27}
\end{equation}%
Obviously, for big $\>\mu $ $\>$Eq. (\ref{equation 27}) has an infinite set
of roots. Thus, the proof of theorem is completed.
\end{proof}

\section{\textbf{Asymptotic Formulas for Eigenvalues and Eigenfunctions}}

Now we begin to study asymptotic properties of eigenvalues and
eigenfunctions. In the following we shall assume that $\>\mu \>$ is
sufficiently big. From (\ref{equation 11}) and (\ref{equation 18}), we
obtain 
\begin{equation}
w_{1}(x,\>\mu )=O(1)\quad \mbox{on}\quad \lbrack 0,\>h_{1}].
\label{equation 28}
\end{equation}%
By (\ref{equation 12}) and (\ref{equation 19}), leads to 
\begin{equation}
w_{2}(x,\>\mu )=O(1)\quad \mbox{on}\quad \lbrack h_{1},\>h_{2}].
\label{equation 29}
\end{equation}%
By (\ref{equation 13}) and (\ref{equation 20}), leads to 
\begin{equation}
w_{3}(x,\>\mu )=O(1)\quad \mbox{on}\quad \lbrack h_{2},\>\pi ].
\label{equation 30}
\end{equation}%
The existence and continuity of the derivatives$\>$ $w_{1\mu }^{\prime
}(x,\>\mu )\>$ for $\>0\leq x\leq h_{1},\>|\mu |<\infty $, $w_{2\mu
}^{\prime }(x,\>\mu )\>$ for $\>h_{1}\leq x\leq h_{2},\>|\mu |<\infty $ and $%
w_{3\mu }^{\prime }(x,\>\mu )\>$ for $\>h_{2}\leq x\leq \pi ,\>|\mu |<\infty 
$ follows from Theorem 1.4.1 in \cite{Norkin 2}.

\begin{lemma}
The following holds true:
\end{lemma}

\begin{align}
w_{1\mu }^{\prime }(x,\mu )& =O(1),\quad x\in \lbrack 0,h_{1}]\text{,}
\label{equation 31} \\
w_{2\mu }^{\prime }(x,\mu )& =O(1),\quad x\in \lbrack \>h_{1},h_{2}]\text{,}
\label{equation 32} \\
w_{3\mu }^{\prime }(x,\mu )& =O(1),\quad x\in \lbrack \>h_{2},\pi ]\text{.}
\label{equation 33}
\end{align}

\begin{proof}
By differentiating (\ref{equation 13}) with respect to $\mu $, we get, by (%
\ref{equation 31}) and (\ref{equation 32})%
\begin{equation}
w_{3\mu }^{\prime }(x,\mu )=-\frac{1}{\mu }\dint\limits_{h_{2}}^{x}q(\tau
)\sin \mu (x-\tau )w_{3\mu }^{\prime }\left( \tau -\Delta \left( \tau
\right) ,\mu \right) +R(x,\mu ),\text{ \ }(\left\vert R(x,\mu )\right\vert
\leq R_{0})\text{.}  \label{equation 34}
\end{equation}%
Let $D_{\mu }=\max_{[h_{2},\pi ]}\left\vert w_{3\mu }^{\prime }(x,\mu
)\right\vert $. Then the existance of $D_{\mu }$ follows from continuity of
derivation for $x\in \lbrack h_{2},\pi ]$. From (\ref{equation 34})%
\begin{equation*}
D_{\mu }\leq \frac{1}{\mu }q_{3}D_{\mu }+R_{0}\text{.}
\end{equation*}%
Now let $\mu \geq 2q_{3}$. Then $D_{\mu }\leq 2R_{0}$ and the validity of
the asymptotic formula (\ref{equation 33}) follows. Formulas (\ref{equation
31}) and (\ref{equation 32}) may be proved analogically.
\end{proof}

Let $N$ be a natural number. We shall say that the number $\mu $ takes part
near the number $\left( N+\frac{1}{4}\right) ^{2}$ if $\left\vert N+\frac{1}{%
4}-\mu \right\vert <\frac{1}{4}$

\begin{theorem}
Let $N$ be a natural number. For each sufficiently big values of $N$ there
is exactly one eigenvalue of the problem (\ref{equation 1})-(\ref{equation 7}%
)\ near $N+\frac{1}{4}$.
\end{theorem}

\begin{proof}
We consider the expression which is denoted by $\>O(1)$ in Equation (\ref%
{equation 27}):%
\begin{equation*}
-\int\limits_{0}^{h_{1}}q(\tau )\sin \mu (\pi -\tau )w_{1}(\tau -\Delta
(\tau ),\mu )d\tau
\end{equation*}%
\begin{equation*}
-\delta \int\limits_{h_{1}}^{h_{2}}q(\tau )\sin \mu (\pi -\tau )w_{2}(\tau
-\Delta (\tau ),\mu )d\tau
\end{equation*}%
\begin{equation*}
-\theta \delta \int\limits_{h_{2}}^{\pi }q(\tau )\sin \mu (\pi -\tau
)w_{3}(\tau -\Delta (\tau ),\mu )d\tau
\end{equation*}%
\begin{equation*}
-\left( \sin \mu \pi +\cos \mu \pi \right) -\frac{1}{\mu }%
\int\limits_{0}^{h_{1}}q(\tau )\cos \mu (\pi -\tau )w_{1}(\tau -\Delta (\tau
),\mu )d\tau
\end{equation*}%
\begin{equation*}
-\frac{\delta }{\mu }\int\limits_{h_{1}}^{h_{2}}q(\tau )\cos \mu (\pi -\tau
)w_{2}(\tau -\Delta (\tau ),\mu )d\tau
\end{equation*}%
\begin{equation*}
-\frac{\delta \theta }{\mu }\int\limits_{h_{2}}^{\pi }q(\tau )\cos \mu (\pi
-\tau )w_{3}(\tau -\Delta (\tau ),\mu )d\tau .
\end{equation*}%
If formulas (\ref{equation 28})-(\ref{equation 33}) are taken into
consideration, it can be shown by differentiation with respect to $\>\mu $
that for big$\>$ $\mu $ $\>$this expression has bounded derivative. We shall
show that, for big $\>N$, only one root (\ref{equation 27}) lies near to
each $N$. We consider the function $\>\phi (\mu )=\mu \sin \left( \mu \pi -%
\frac{\pi }{4}\right) +O(1)$. Its derivative, which has the form $\>\phi
^{\prime }(\mu )=\sin \left( \mu \pi -\frac{\pi }{4}\right) +\mu \pi \cos
\left( \mu \pi -\frac{\pi }{4}\right) +O(1)$, does not vanish for $\>\mu \>$
close to$\>N+\frac{1}{4}\>$for sufficiently big$\>$ $N$. Thus our assertion
follows by Rolle's Theorem.
\end{proof}

Let $\>N\>$ be sufficiently big. In what follows we shall denote by $\mu
_{n}^{2}$ the eigenvalue of the problem (\ref{equation 1})-(\ref{equation 7}%
). We set $\mu _{N}=N+\frac{1}{4}+\delta _{N}$. Then from (\ref{equation 27}%
) it follows that $\delta _{N}=O\left( \frac{1}{N}\right) $. Consequently%
\begin{equation}
\mu _{N}=N+\frac{1}{4}+O\bigl ({\frac{1}{N}}\bigr ),  \label{equation 35}
\end{equation}%
Formula (\ref{equation 35}) make it possible to obtain asymptotic
expressions for eigenfunction of the problem (\ref{equation 1})-(\ref%
{equation 7}). From (\ref{equation 11}), (\ref{equation 21}) and (\ref%
{equation 28}), we get 
\begin{equation}
w_{1}(x,\>\mu )=\sqrt{2}\cos \left( \frac{\pi }{4}+\mu x\right) +O\bigl ({%
\frac{1}{\mu }}\bigr ),  \label{equation 36}
\end{equation}%
and%
\begin{equation}
w_{1}^{\prime }(x,\>\mu )=-\mu \sqrt{2}\sin \left( \frac{\pi }{4}+\mu
x\right) +O\bigl ({1}\bigr ).  \label{equation 37}
\end{equation}%
From expressions of (\ref{equation 12}), (\ref{equation 22}), (\ref{equation
27}), (\ref{equation 29}) and (37), we easily see that%
\begin{equation}
w_{2}(x,\>\mu )=\frac{\sqrt{2}}{\delta }\cos \left( \frac{\pi }{4}+\mu
x\right) +O\bigl ({\frac{1}{\mu }}\bigr ).  \label{equation 38}
\end{equation}%
and%
\begin{equation}
w_{3}(x,\>\mu )=\frac{\sqrt{2}}{\delta \theta }\cos \left( \frac{\pi }{4}%
+\mu x\right) +O\bigl ({\frac{1}{\mu }}\bigr ).  \label{equation 39}
\end{equation}%
By substituting (\ref{equation 35}) into (\ref{equation 36}), (\ref{equation
38}) and (\ref{equation 39}), we find that%
\begin{align*}
U_{1N}& =w_{1}\left( x,\mu _{N}\right) =\sqrt{2}\cos \left( \frac{\pi }{4}%
+\left( N+\frac{1}{4}\right) x\right) +O\bigl ({\frac{1}{N}}\bigr )\text{,}
\\
U_{2N}& =w_{2}\left( x,\mu _{N}\right) =\frac{\sqrt{2}}{\delta }\cos \left( 
\frac{\pi }{4}+\left( N+\frac{1}{4}\right) x\right) +O\bigl ({\frac{1}{N}}%
\bigr )\text{,} \\
U_{3N}& =w_{3}\left( x,\mu _{N}\right) =\frac{\sqrt{2}}{\delta \theta }\cos
\left( \frac{\pi }{4}+\left( N+\frac{1}{4}\right) x\right) +O\bigl ({\frac{1%
}{N}}\bigr )\text{.}
\end{align*}%
Hence the eigenfunctions $\>u_{N}(x)\>$ have the following asymptotic
representation:%
\begin{equation*}
U_{N}(x)=\left\{ 
\begin{array}{ccc}
U_{1N}\text{,} & \mbox{for} & x\in \lbrack 0,h_{1}), \\ 
U_{2N}\text{,} & \mbox{for} & x\in \left( h_{1},h_{2}\right) , \\ 
U_{3N}\text{,} & \mbox{for} & x\in \left( h_{2},\pi \right] .%
\end{array}%
\right.
\end{equation*}%
Under some additional conditions the more exact asymptotic formulas which
depend upon the retardation may be obtained. Let us assume that the
following conditions are fulfilled:

a.) The derivatives $q^{\prime}(x)$ and $\Delta^{\prime\prime}(x)$ exist and
are bounded in $\left[ 0,h_{1}\right) \cup\left( h_{1},h_{2}\right)
\cup\left( h_{2},\pi\right] $ and have finite limits $q^{\prime}\left(
h_{1}\pm0\right) =\lim_{x\rightarrow h_{1}\pm0}q^{\prime}(x),$ $q^{\prime
}\left( h_{2}\pm0\right) =\lim_{x\rightarrow h_{2}\pm0}q^{\prime}(x)$ and $%
\Delta^{\prime\prime}\left( h_{1}\pm0\right) =\lim_{x\rightarrow h_{1}\pm
0}\Delta^{\prime\prime}(x),$ $\Delta^{\prime\prime}\left( h_{2}\pm0\right)
=\lim_{x\rightarrow h_{2}\pm0}\Delta^{\prime\prime}(x)$, respectively.

b.) $\Delta^{\prime}\left( x\right) \leq1$ in $\left[ 0,h_{1}\right)
\cup\left( h_{1},h_{2}\right) \cup\left( h_{2},\pi\right] ,$ $\Delta\left(
0\right) =0,$ $\lim_{x\rightarrow h_{1}+0}\Delta\left( x\right) =0$ and $%
\lim_{x\rightarrow h_{2}+0}\Delta\left( x\right) =0.$

It is easy to see that, using b.)%
\begin{align}
x-\Delta \left( x\right) & \geq 0,\text{ }x\in \left[ 0,h_{1}\right) \text{,}
\label{equation 40} \\
x-\Delta \left( x\right) & \geq h_{1},\text{ }x\in \left( h_{1},h_{2}\right) 
\text{,}  \label{equation 41}
\end{align}%
and%
\begin{equation}
x-\Delta \left( x\right) \geq h_{2},\text{ }x\in \left( h_{2},\pi \right]
\label{equation 42}
\end{equation}

are obtained.

By (\ref{equation 36}), (\ref{equation 38}), (\ref{equation 39}), (\ref%
{equation 40}), (\ref{equation 41}) and (\ref{equation 42}) we have%
\begin{equation}
w_{1}\left( \tau -\Delta \left( \tau \right) ,\mu \right) =\sqrt{2}\cos
\left( \frac{\pi }{4}+\mu \left( \tau -\Delta \left( \tau \right) \right)
\right) +O\left( \frac{1}{\mu }\right) \text{,}  \label{equation 43}
\end{equation}%
\begin{equation}
w_{2}\left( \tau -\Delta \left( \tau \right) ,\mu \right) =\frac{\sqrt{2}}{%
\delta }\cos \left( \frac{\pi }{4}+\mu \left( \tau -\Delta \left( \tau
\right) \right) \right) +O\left( \frac{1}{\mu }\right)  \label{equation 44}
\end{equation}%
and%
\begin{equation}
w_{3}\left( \tau -\Delta \left( \tau \right) ,\mu \right) =\frac{\sqrt{2}}{%
\delta \theta }\cos \left( \frac{\pi }{4}+\mu \left( \tau -\Delta \left(
\tau \right) \right) \right) +O\left( \frac{1}{\mu }\right)
\label{equation 45}
\end{equation}%
on $\left[ 0,h_{1}\right) ,$ $\left( h_{1},h_{2}\right) $ and $\left(
h_{2},\pi \right] $ respectively.

Under the conditions a.) and b.) the following formulas:%
\begin{equation}
\left. 
\begin{array}{c}
\int_{0}^{x}q\left( \tau \right) \cos \mu \left( 2\tau -\Delta \left( \tau
\right) \right) d\tau =O\left( \frac{1}{\mu }\right) \text{,} \\ 
\int_{0}^{x}q\left( \tau \right) \sin \mu \left( 2\tau -\Delta \left( \tau
\right) \right) d\tau =O\left( \frac{1}{\mu }\right)%
\end{array}%
\right\}  \label{equation 46}
\end{equation}%
can be proved by the same technique in Lemma 3.3.3 in \cite{Norkin 2}.

Putting the expressions (\ref{equation 43}), (\ref{equation 44}) and (\ref%
{equation 45}) into (\ref{equation 26}), and then using (\ref{equation 46}),
after long operations we have%
\begin{equation*}
\tan \pi \left( \mu -\frac{1}{4}\right) =-\frac{1}{2\mu }\left[
2+\int_{0}^{\pi }q\left( \tau \right) \cos \mu \Delta \left( \tau \right)
d\tau \right] +O\left( \frac{1}{\mu ^{2}}\right) .
\end{equation*}%
Again, if we take $\mu _{N}=N+\frac{1}{4}+\delta _{N},$ we obtain for
sufficiently big $N,$%
\begin{equation*}
\delta _{N}=-\frac{2}{\pi \left( 4N+1\right) }\left( 2+\int_{0}^{\pi
}q\left( \tau \right) \cos \left( \left( N+\frac{1}{4}\right) \Delta \left(
\tau \right) \right) d\tau \right) +O\left( \frac{1}{N^{2}}\right)
\end{equation*}%
and finally%
\begin{equation}
\mu _{N}=N+\frac{1}{4}-\frac{2}{\pi \left( 4N+1\right) }\left(
2+\int_{0}^{\pi }q\left( \tau \right) \cos \left( \left( N+\frac{1}{4}%
\right) \Delta \left( \tau \right) \right) d\tau \right) +O\left( \frac{1}{%
N^{2}}\right) \text{.}  \label{equation 47}
\end{equation}%
Thus, we have proven the following theorem:

\begin{theorem}
If conditions a.) and b.) are satisfied then, the eigenvalues $\mu _{N}^{2}$
of the problem (\ref{equation 1})-(\ref{equation 7}) have the (\ref{equation
47}) asymptotic formula for $N\rightarrow \infty $.
\end{theorem}

Now, we may obtain sharper asymptotic formulas for the eigenfunctions. From (%
\ref{equation 11}), (\ref{equation 43}), (\ref{equation 46}) and replacing $%
\mu $ by $\mu _{N}$ we have%
\begin{equation*}
u_{1N}(x)=\cos \left( \left( N+\frac{1}{4}\right) x\right) \left[ 1+\frac{1}{%
2N}\left[ \int_{0}^{x}q\left( \tau \right) \sin \left( \left( N+\frac{1}{4}%
\right) \Delta \left( \tau \right) \right) d\tau \right. \right.
\end{equation*}%
\begin{equation*}
\left. \left. -\int_{0}^{x}q\left( \tau \right) \cos \left( \left( N+\frac{1%
}{4}\right) \Delta \left( \tau \right) \right) d\tau \right] \right] +\cos
\left( \left( N+\frac{1}{4}\right) x\right)
\end{equation*}%
\begin{equation*}
\times \left[ \frac{2x}{\pi \left( 4N+1\right) }\left[ 2+\int_{0}^{\pi
}q\left( \tau \right) \cos \left( \left( N+\frac{1}{4}\right) \Delta \left(
\tau \right) \right) d\tau \right] \right]
\end{equation*}%
\begin{equation*}
-\sin \left( \left( N+\frac{1}{4}\right) x\right) \left[ 1+\frac{1}{2N}\left[
\int_{0}^{x}q\left( \tau \right) \sin \left( \left( N+\frac{1}{4}\right)
\Delta \left( \tau \right) \right) d\tau \right. \right.
\end{equation*}%
\begin{equation*}
\left. \left. -\int_{0}^{x}q\left( \tau \right) \cos \left( \left( N+\frac{1%
}{4}\right) \Delta \left( \tau \right) \right) d\tau \right] \right] +\sin
\left( \left( N+\frac{1}{4}\right) x\right)
\end{equation*}%
\begin{equation}
\times \left[ \frac{2x}{\pi \left( 4N+1\right) }\left[ 2+\int_{0}^{\pi
}q\left( \tau \right) \cos \left( \left( N+\frac{1}{4}\right) \Delta \left(
\tau \right) \right) d\tau \right] \right] +O\left( \frac{1}{N^{2}}\right) 
\text{.}  \label{equation 48}
\end{equation}%
From (\ref{equation 12}), (\ref{equation 44}), (\ref{equation 46}) and
replacing $\mu $ by $\mu _{N}$ we have

\begin{equation*}
u_{2N}(x)=\frac{1}{\delta }\left\{ \cos \left( \left( N+\frac{1}{4}\right)
x\right) \left[ 1+\frac{1}{2N}\left[ \int_{0}^{x}q\left( \tau \right) \sin
\left( \left( N+\frac{1}{4}\right) \Delta \left( \tau \right) \right) d\tau
\right. \right. \right.
\end{equation*}%
\begin{equation*}
\left. \left. -\int_{0}^{x}q\left( \tau \right) \cos \left( \left( N+\frac{1%
}{4}\right) \Delta \left( \tau \right) \right) d\tau \right] \right] +\cos
\left( \left( N+\frac{1}{4}\right) x\right)
\end{equation*}%
\begin{equation*}
\times \left[ \frac{2x}{\pi \left( 4N+1\right) }\left[ 2+\int_{0}^{\pi
}q\left( \tau \right) \cos \left( \left( N+\frac{1}{4}\right) \Delta \left(
\tau \right) \right) d\tau \right] \right]
\end{equation*}%
\begin{equation*}
-\sin \left( \left( N+\frac{1}{4}\right) x\right) \left[ 1+\frac{1}{2N}\left[
\int_{0}^{x}q\left( \tau \right) \sin \left( \left( N+\frac{1}{4}\right)
\Delta \left( \tau \right) \right) d\tau \right. \right.
\end{equation*}%
\begin{equation*}
\left. \left. -\int_{0}^{x}q\left( \tau \right) \cos \left( \left( N+\frac{1%
}{4}\right) \Delta \left( \tau \right) \right) d\tau \right] \right] +\sin
\left( \left( N+\frac{1}{4}\right) x\right)
\end{equation*}%
\begin{equation}
\left. \times \left[ \frac{2x}{\pi \left( 4N+1\right) }\left[
2+\int_{0}^{\pi }q\left( \tau \right) \cos \left( \left( N+\frac{1}{4}%
\right) \Delta \left( \tau \right) \right) d\tau \right] \right] \right\}
+O\left( \frac{1}{N^{2}}\right) \text{.}  \label{equation 49}
\end{equation}

From (\ref{equation 13}), (\ref{equation 45}), (\ref{equation 46}), and
replacing $\mu $ by $\mu _{N}$, after long operations we have

\begin{equation*}
u_{3N}(x)=\frac{1}{\delta \theta }\left\{ \cos \left( \left( N+\frac{1}{4}%
\right) x\right) \left[ 1+\frac{1}{2N}\left[ \int_{0}^{x}q\left( \tau
\right) \sin \left( \left( N+\frac{1}{4}\right) \Delta \left( \tau \right)
\right) d\tau \right. \right. \right.
\end{equation*}%
\begin{equation*}
\left. \left. -\int_{0}^{x}q\left( \tau \right) \cos \left( \left( N+\frac{1%
}{4}\right) \Delta \left( \tau \right) \right) d\tau \right] \right] +\cos
\left( \left( N+\frac{1}{4}\right) x\right)
\end{equation*}%
\begin{equation*}
\times \left[ \frac{2x}{\pi \left( 4N+1\right) }\left[ 2+\int_{0}^{\pi
}q\left( \tau \right) \cos \left( \left( N+\frac{1}{4}\right) \Delta \left(
\tau \right) \right) d\tau \right] \right]
\end{equation*}%
\begin{equation*}
-\sin \left( \left( N+\frac{1}{4}\right) x\right) \left[ 1+\frac{1}{2N}\left[
\int_{0}^{x}q\left( \tau \right) \sin \left( \left( N+\frac{1}{4}\right)
\Delta \left( \tau \right) \right) d\tau \right. \right.
\end{equation*}%
\begin{equation*}
\left. \left. -\int_{0}^{x}q\left( \tau \right) \cos \left( \left( N+\frac{1%
}{4}\right) \Delta \left( \tau \right) \right) d\tau \right] \right] +\sin
\left( \left( N+\frac{1}{4}\right) x\right)
\end{equation*}%
\begin{equation}
\left. \times \left[ \frac{2x}{\pi \left( 4N+1\right) }\left[
2+\int_{0}^{\pi }q\left( \tau \right) \cos \left( \left( N+\frac{1}{4}%
\right) \Delta \left( \tau \right) \right) d\tau \right] \right] \right\}
+O\left( \frac{1}{N^{2}}\right) \text{.}  \label{equation 50}
\end{equation}

Thus, we have proven the following theorem:

\begin{theorem}
If conditions a.) and b.) are satisfied then, the eigenfunctions $U_{N}(x)$
of the problem (\ref{equation 1})-(\ref{equation 7}) have the following
asymptotic formula for $N\rightarrow \infty $:%
\begin{equation*}
U_{N}(x)=\left\{ 
\begin{array}{cc}
U_{1N}(x), & x\in \left[ 0,h_{1}\right) , \\ 
U_{2N}(x), & x\in \left( h_{1},h_{2}\right) , \\ 
U_{3N}(x), & x\in \left( h_{2},\pi \right]%
\end{array}%
\right.
\end{equation*}%
where $U_{1N}(x),$ $U_{2N}(x)$ and $U_{3N}(x)$ determined as in (\ref%
{equation 48}), (\ref{equation 49}) and (\ref{equation 50}) respectively.
\end{theorem}



\end{document}